\documentclass[smallcondensed]{svjour3} 

\usepackage{amsmath,amssymb,graphicx,amsfonts,bm,algorithm,hyperref,xcolor}
\usepackage[noend]{algorithmic}

\journalname{Journal of Scientific Computing}
\smartqed

\newcommand{\J}[1]{{J}_{#1}}

\newcommand{\vv}[1]{ v_{#1}^{(n)} }
\newcommand{\ds}{\displaystyle}
\newcommand{\vt}[1]{{v_{#1}^{(n)}}^\top}

\newtheorem{assumption}{Assumption}

\begin{document}

\title{Accelerated high-index saddle dynamics method for searching high-index saddle points}

\author{
	Yue Luo \and Xiangcheng Zheng \and Lei Zhang 
}

\titlerunning{Accelerated high-index saddle dynamics}

\authorrunning{Y. Luo, X. Zheng, L. Zhang}

\institute{Y. Luo \at
	Beijing International Center for Mathematical Research, Peking University, Beijing 100871\\
	\email{moonluo@pku.edu.cn}
	\and 
	X. Zheng \at 
	School of Mathematics, Shandong University, Jinan 250100, China\\
	\email{xzheng@sdu.edu.cn}
        \and
	L. Zhang \at
	Beijing International Center for Mathematical Research, Center for Machine Learning Research, Center for Quantitative Biology, Peking University, Beijing 100871, China\\	\email{zhangl@math.pku.edu.cn}
}

\date{Received: date / Accepted: date}

\maketitle
\begin{abstract}
The high-index saddle dynamics (HiSD) method [J. Yin, L. Zhang, and P. Zhang, {\it SIAM J. Sci. Comput., }41 (2019), pp.A3576-A3595] serves as an efficient tool for computing index-$k$ saddle points and constructing solution landscapes. Nevertheless, the conventional HiSD method often encounters  slow convergence rates on ill-conditioned problems. To address this challenge, we propose an accelerated high-index saddle dynamics (A-HiSD) by incorporating the heavy ball method. We prove the linear stability theory of the continuous A-HiSD, and subsequently  estimate the local convergence rate for the discrete A-HiSD. Our analysis demonstrates that the A-HiSD method exhibits a faster convergence rate compared to the conventional HiSD method, especially when dealing with ill-conditioned problems. We also perform various numerical experiments including the loss function of neural network to substantiate the effectiveness and acceleration of the A-HiSD method.



 \end{abstract}

\section{Introduction}
Searching saddle points in various complex systems has attracted attentions in many scientific fields, such as finding critical nuclei and transition pathways in phase transformations \cite{cheng2010nucleation,Han2019transition,samanta2014microscopic,wang2010phase,zhang2007morphology},  protein folding \cite{burton1997energy,wales2003energy}, Lenard-Jones clusters \cite{cai2018single,wales1997global}, loss landscape analysis in deep learning \cite{NIPS2014_17e23e50,NEURIPS2019_a4ee59dd,NEURIPS2021_7cc532d7}, excited states in Bose-Einstein condensates \cite{yin2023revealing}, etc. A typical example is that the transition state is characterized as the index-1 saddle point of the energy function, i.e. a critical point where the Hessian contains only one negative eigenvalue. Complex structure of saddle points makes it more challenging to design saddle searching algorithms than optimization methods.

 In general, there are two classes of approaches for finding saddle points: path-finding methods and surface walking methods. The former class is suitable for locating index-1 saddle points when initial and final states on the energy surface are available, such as the string method \cite{weinan2007simplified} and the nudged elastic band method \cite{henkelman2000improved}. It searches a string connecting the initial and final points to find the minimum energy path. The latter class, which is similar to optimization methods,  starts from a single point and utilizes first-order or second-order derivatives to search saddle points. Representative methods include the gentlest ascent method \cite{gad}, the eigenvector-following method \cite{cerjan1981finding}, the minimax method \cite{li2001minimax}, the activation-relaxation techniques \cite{cances2009some}, and the dimer type method \cite{gould2016dimer,henkelman1999dimer,zhangdu2012,zhang2016optimization}. Some of those methods have been extended to the calculation of index-$k$ ($k\in\mathbb{N}$) saddle points \cite{gao2015iterative,gusimplified,quapp2014locating,yin2019high}.

Recently, high-index saddle dynamics (HiSD) method \cite{yin2019high} severs as a powerful tool for the calculation of high-index saddle points and the construction of solution landscape when combined with the downward and upward algorithms \cite{YinPRL,jcm2023}. It has been successfully extended for the saddle point calculations in non-gradient systems \cite{yin2021searching} and constrained systems \cite{yin2022constrained,scm2023}. The HiSD method and solution landscape have been widely used in physical and engineering problems, including finding the defect configurations in liquid crystals \cite{Han2021,han2021solution,shi2023nonlinearity,wang2021modeling,yin2022solution}, nucleation in quasicrystals \cite{Yin2020nucleation}, morphologies of confined diblock copolymers \cite{xu2021solution}, excited states in rotating Bose-Einstein condensates \cite{yin2023revealing}.

It was demonstrated in \cite{yin2019high} that a linearly stable steady state of the HiSD corresponds to an index$-k$ saddle point, that is, a critical point of energy function where the Hessian has and only has $k$ negative eigenvalues. Furthermore, \cite{luo2022convergence} proved that the local linear convergence rate of the discrete HiSD method is $1-\mathcal{O}(\frac{1}{\kappa})$, where $\kappa\geq 1$ is the  ratio of absolute values of maximal and minimal eigenvalues of the Hessian at the saddle point, reflecting the curvature information around this saddle point. In practice, many examples show that HiSD method suffers from the slow convergence rate if the problem is ill-conditioned, that is, $\kappa$ is large. Hence, it is meaningful to design reliable acceleration mechanisms for the HiSD method to treat ill-conditioned problems.  

The heavy ball method (HB) is a simple but efficient acceleration strategy that is widely applied in the field of optimization. It integrates the information of previous search direction, i.e. the momentum term $\gamma(x^{(n)} - x^{(n-1)})$, in the current step of optimization. Polyak pointed out that invoking the HB method could improve the local linear convergence rate of the gradient descent method \cite{polyak1964some}. Till now, the HB method has become a prevalent acceleration strategy in various situations. For instance, many adaptive gradient methods in deep learning such as the Adam \cite{kingma2014adam}, AMSGrad \cite{reddi2019convergence} and AdaBound \cite{Luo2019AdaBound} adopt the idea of HB method. Nguyen et al. \cite{nguyen} proposed an accelerated residual method by combing the HB method with an extra gradient descent step and numerically verified the efficiency of this method. 

In this paper, we propose an accelerated high-index saddle dynamics (A-HiSD) method based on the HB method to enhance the convergence rate of the conventional HiSD method, especially  for ill-conditioned problems. At the continuous level, we prove that a linearly stable steady state of the A-HiSD is an index-$k$ saddle point, showcasing the effectiveness of A-HiSD in searching saddle points. The key ingredient of our proof relies on technical eigenvalue analysis of the Jacobian operator of the dynamical system. Subsequently, we propose and analyze the discrete scheme of A-HiSD, indicating a faster local convergence rate in comparison with the original HiSD method. Novel induction hypothesis method and matrix analysis methods are applied to accommodate the locality of the assumptions in the problem. 
To substantiate the effectiveness and acceleration capabilities of A-HiSD, we conduct a series of numerical experiments, including benchmark examples and ill-conditioned problems such as the loss function of neural networks. These results provide mathematical and numerical analysis support for application of the A-HiSD in saddle point search for ill-conditioned problems.

The paper is organized as follows: in Section 2 we review the HiSD method briefly and present a detailed description of the A-HiSD method. Stability analysis of the A-HiSD dynamical system is discussed in Section 3. Local convergence rate estimates of the proposed algorithm are proved in Section 4. Several experiments are presented to substantiate the acceleration of the proposed method in Section 5. Conclusions and further discussions are given in Section 6.   

\section{Formulation of A-HiSD}
Let $E:\mathbb{R}^d\rightarrow \mathbb{R}$ be a real valued, at least twice differentiable energy function defined on $d-$dimensional Euclidean space. We denote $\|\cdot\|_2:\mathbb{R}^{d\times d}\rightarrow \mathbb{R}$ for the operator norm of $d\times d$ real matrices, i.e.  
$\|A\|_2=\max_{\|x\|_2=1}\|Ax\|_2,~~x\in\mathbb{R}^d$ and the vector 2-norm $\|x\|^2_2: = \sum_{i=1}^d x_i^2.$ We also denote $x^*$ as the non-degenerate index-$k$ saddle point, i.e., the critical point of $E$ ($\nabla E(x^*)=0$) where the Hessian $\nabla^2 E(x^*)$ is invertible and only has $k$ negative eigenvalues, according to the definition of Morse index \cite{milnor2016morse}.

\subsection{Review of HiSD}
HiSD method provides a powerful instrument to search any-index saddle points. The HiSD for an index-$k$ ($1\leq\ k \in\mathbb{N}$) saddle point of the energy function $E(x)$ reads \cite{yin2019high}
\begin{equation}
	\left\{
	\begin{aligned}
		\dot{x} & =- \beta\Big(I-2\sum_{i=1}^kv_iv_i^\top\Big)\nabla E(x),\\
		\dot{v_i} &=- \zeta\Big(I-v_iv_i^\top-2\sum_{j=1}^{i-1}v_jv_j^\top\Big)\nabla^2E(x)v_i,\ 1\leq i \leq k,
	\end{aligned}
	\right.
	\label{saddle_dyna}
\end{equation}
where $\beta$ and $\zeta$ are positive relaxation parameters.
The continuous formulation (\ref{saddle_dyna}) is discussed in details in \cite{yin2019high}, and the convergence rate and error estimates of the discrete algorithm are analyzed in \cite{luo2022convergence,zhang2022sinum}. 

\begin{algorithm}[H]
	\caption{HiSD for an index-$k$ saddle point}
	\begin{algorithmic}
        \STATE{{\bf Input:} $k\in\mathbb{N}$, $x^{(0)}\in\mathbb{R}^d$, $\big\{\hat v_i^{(0)}\big\}_{i=1}^k\subset \mathbb{R}^d$ satisfying ${{}\hat v_i^{(0)}}^\top \hat v_j^{(0)}=\delta_{ij}$.}
		\FOR{$n=0,1,\ldots,T-1$}
		\STATE {$\ds x^{(n+1)} = x^{(n)} - \beta\Big(I-2\sum_{i=1}^k\hat v_i^{(n)}{{}\hat v_i^{(n)}}^\top\Big)\nabla E(x^{(n)})$; }
		\STATE {$\big\{\hat v_i^{(n+1)}\big\}_{i=1}^k = \text{EigenSol}\big( \big\{\hat v_i^{(n)}\big\}_{i=1}^k, \nabla^2 E(x^{(n+1)}) \big)$.}
		\ENDFOR
        \RETURN $x^{(T)}$
	\end{algorithmic}
 \label{alg_ori}
\end{algorithm}
The discrete scheme of HiSD method is shown in Algorithm \ref{alg_ori}. ``EigenSol'' represents some eigenvector solver with initial values $\big\{\hat v_i^{(n)}\big\}_{i=1}^k$ to compute eigenvectors corresponding to the first $k$ smallest eigenvalues of $\nabla^2 E(x^{(n+1)})$, such as the simultaneous Rayleigh-quotient iterative minimization method (SIRQIT) \cite{longsine1980simultaneous} and locally optimal block preconditioned conjugate gradient (LOBPCG) method \cite{knyazev2001toward}.

\subsection{Algorithm of A-HiSD}
Inspired by the simple implementation and the efficient acceleration phenomenon of the HB method, we propose the A-HiSD method described in Algorithm \ref{hhisd}.
\begin{algorithm}[H]
	\caption{A-HiSD for an index-$k$ saddle point}
	\label{hhisd}
	\begin{algorithmic}
        \STATE{ {\bf Input:} 
	$k\in\mathbb{N}$, $x^{(0)}\in\mathbb{R}^d$, orthonormal vectors $\big\{v_i^{(0)}\big\}_{i=1}^k\subset \mathbb{R}^d$.}
		\STATE {Set $x^{(-1)}=x^{(0)}$.}
		\FOR{$n=0,1,\ldots,T-1$}
		\STATE {$\displaystyle x^{(n+1)} = x^{(n)} - \beta\Big(I-2\sum_{i=1}^kv_i^{(n)}{{}v_i^{(n)}}^\top\Big)\nabla E(x^{(n)}) + \gamma(x^{(n)}-x^{(n-1)})$;} 	
		\STATE {$\big\{v_i^{(n+1)}\big\}_{i=1}^k = \text{EigenSol}\big( \big\{v_i^{(n)}\big\}_{i=1}^k, \nabla^2 E(x^{(n+1)}) \big)$.}
		\ENDFOR
		\RETURN $x^{(T)}$
	\end{algorithmic}
\end{algorithm}
 
From the perspective of algorithm implementation, the only difference between A-HiSD method and HiSD method is the update formula of the position variable $x^{(n)}$. The A-HiSD method integrates the information of previous search direction in the current step, i.e. the momentum term $\gamma(x^{(n)} - x^{(n-1)})$, motivated by the HB method.  Due to this modification, a significant acceleration is expected with the proper choice of the step size parameter $\beta$ and momentum parameter $\gamma$. We will perform numerical analysis and numerical experiments in the following sections.

From the continuous level, the HB method is highly connected with a first-order ordinary differential equation system due to the presence of multiple steps \cite{nesterov2003introductory}. Following this idea, we introduce the continuous A-HiSD system for an index-$k$ saddle point: 

\begin{equation}
	\left\{
	\begin{aligned}
		&\dot{x} = m,\\
		&\dot{m} = -\alpha_1m - \alpha_2\Big(I-2\sum_{i=1}^kv_iv_i^\top\Big)\nabla E(x),\\ 
		&\dot{v}_i = -\zeta\Big(I - v_iv_i^\top - 2\sum_{j=1}^{i-1}v_jv_j^\top\Big)H(x,v_i,l),\ i=1,2,...,k, \\
		&\dot{l} = -l,
	\end{aligned}
	\right.
	\label{A-HiSD_dyna}
\end{equation}
where $H(x,v,l)$ is the dimer approximation for the Hessian-vector product $\nabla^2E(x)v$ \cite{henkelman1999dimer,yin2019high}
\begin{equation}
	\nabla^2 E(x)v\approx H(x,v,l) = \frac{1}{2l}\left[\nabla E(x+lv) - \nabla E (x-lv)\right].
	\label{dimer}
\end{equation}
Here $\alpha_1$, $\alpha_2$ and $\zeta$ are positive relaxation parameters. 
\begin{remark}
To see the relations between the continuous and discrete formulations of the position variable in (\ref{A-HiSD_dyna}) and the Algorithm \ref{hhisd}, respectively, we first combine the first two equations in (\ref{A-HiSD_dyna}) to obtain
$$\ddot{x} = -\alpha_1\dot{x} - \alpha_2\Big(I-2\sum_{i=1}^kv_iv_i^\top\Big)\nabla E(x). $$
Then we discretize this equation with the step size $\Delta t$ to obtain
$$\frac{x^{(n+1)}-2x^{(n)}+x^{(n-1)}}{(\Delta t)^2}=-\alpha_1\frac{x^{(n)}-x^{(n-1)}}{\Delta t}- \alpha_2\Big(I-2\sum_{i=1}^kv_i^{(n)}{v_i^{(n)}}^\top\Big)\nabla E(x^{(n)}). $$
If we set $\beta=\alpha_2(\Delta t)^2$ and $\gamma=1-\alpha_1\Delta t$, the scheme of the position variable in Algorithm \ref{hhisd} is consistent with the above scheme. Therefore, the formulation of the position variable in (\ref{A-HiSD_dyna}) is indeed the continuous version (or the limit case when $\Delta t\rightarrow 0^+$) of the A-HiSD algorithm.
\end{remark}

\section{Linear stability analysis of A-HiSD}
We prove that the linearly stable steady state of (\ref{A-HiSD_dyna}) is exactly an index-$k$ saddle point of $E$, which substantiates the effectiveness of (\ref{A-HiSD_dyna}) (and thus its discrete version in Algorithm \ref{hhisd}) in finding high-index saddle points.
\begin{theorem}
	Assume that $E(x)$ is a $C^3$ function, $\left\{v_i^*\right\}_{i=1}^k\subset\mathbb{R}^d$ satisfy $\|v_i^*\|_2=1$ for $1\leq i\leq k$. The Hessian $\nabla^2 E(x^*)$ is non-degenerate with the eigenvalues $\lambda_1^*<\ldots<\lambda_k^*<0< \lambda_{k+1}^*\leq\ldots\leq \lambda_d^*$, and $\alpha_1,\alpha_2,\zeta>0$ satisfy $\alpha_1^2\leq4\alpha_2\mu^*$ where $\mu^*=\min\{|\lambda_i^*|,\,i=1,\ldots,k\}$. Then the following two statements are equivalent:
\begin{itemize}
\item[(A)] $(x^*,m^*,v_1^*,...,v_k^*,l^*)$ is a linearly stable steady state of the A-HiSD (\ref{A-HiSD_dyna});

\item[(B)] $x^*$ is an index-$k$ saddle point of $E$, $\{v_i^*\}_{i=1}^k$ are eigenvectors of $\nabla^2E(x^*)$ with the corresponding eigenvalues $\{\lambda_i^*\}_{i=1}^k$ and $m^*=l^*=0$.
\end{itemize}
\end{theorem}
\begin{proof}
	To determine the linear stability of A-HiSD (\ref{A-HiSD_dyna}), we calculate its Jacobian operator as follows
\begin{equation}\label{J}
	\J{} = \frac{\partial(\dot{x},\dot{m},\dot{v}_1,\dot{v}_2,...,\dot{v}_k,\dot{l})}{\partial (x,m,v_1,v_2,...,v_k,l)} = 
	\begin{bmatrix}
		0 &I &0 &0 &\ldots &0 & 0\\
		\J{mx} &-\alpha_1 I &\J{m1} &\J{m2}&\ldots &\J{mk} &0\\
		\J{1x} &0 &\J{11} &0 &\ldots &0 &\J{1l}\\
		\J{2x} &0 &\J{21} &\J{22} &\ldots&0 &\J{2l}\\
		\vdots &\vdots &\vdots &\vdots &\vdots &\vdots &\vdots\\
		\J{kx} &\J{km} &\J{k1} &\J{k2} &\cdots &\J{kk} &\J{kl}\\
		0 &0 &0 &0 &0 &0 &-1
	\end{bmatrix},
\end{equation}
where part of the blocks are presented in details as follows
	\begin{align*}
		\J{mx} &= \partial_x \dot{m} = -\alpha_2\mathcal P_k\nabla^2 E(x),\\
		\J{mi} &= \partial_{v_i} \dot{m} = 2\alpha_2\left  (v_i^\top \nabla E(x)I + v_i\nabla E(x)^\top\right ),\\
		J_{ii} & = \partial_{v_i} \dot{v}_i = -\zeta\big(\mathcal P_{i-1} - v_iv_i^\top \big)\partial_{v_i} H(x,v_i,l)+ \zeta(v_i^\top H(x,v_i,l)I + v_iH(x,v_i,l)^\top),\\
		\J{il} & = -\zeta\big(\mathcal P_{i-1} - v_iv_i^\top \big)\partial_l H(x,v_i,l).
	\end{align*}
Here we use the notation $\mathcal P_s:=I-2\sum_{j=1}^s v_jv_j^\top$ for $1\leq s\leq k$ and $\mathcal P_{0} = I$. Derivatives of $H(x,v_i,l)$ with respect to different variables are  
	\begin{align*}
		&\partial_l H(x,v_i,l) = \frac{\nabla^2 E(x+lv_i) + \nabla^2 E(x-lv_i)}{2l}v_i  - \frac{\nabla E(x+lv_i) - \nabla E(x-lv_i)}{2l^2},\\
& \partial_{v_i} H(x,v_i,l)= \frac{\nabla^2E(x+lv_i)+\nabla^2E(x-lv_i)}{2}.
	\end{align*}
 By the smoothness of the energy function $E(x)$ and Taylor expansions, we could directly verify the following limits for future use
\begin{align}\label{Hlim}
&\lim_{l\rightarrow 0^+} H(x,v_i,l)=\nabla^2 E(x)v_i,~~ \lim_{l\rightarrow 0^+}\partial_l H(x,v_i,l) =0,~~\lim_{l\rightarrow 0^+}\partial_{v_i} H(x,v_i,l) =\nabla^2E(x).
\end{align}

We firstly derive (A) from (B). Under the conditions of (B), we have $\nabla E (x^*)=0$, $\nabla^2 E(x^*)v_i^*=\lambda_i^*v_i^*$ for $1\leq i\leq k$ and $m^*=l^*=0$ such that $(x^*, m^*, v_1^*, v_2^*, ..., v_k^*, l^*)$ is a steady state of (\ref{A-HiSD_dyna}). To show the linear stability, we remain to show that all eigenvalues of $J^*$, the Jacobian (\ref{J}) at this steady state, have negative real parts. From $\nabla E(x^*)=0$ and $l^*=0$, we observe that $\J{mi}^*$ and $\J{il}^*$ are zero matrices such that $\J{}^*$ is in the form of 
\begin{equation}
	\J{}^* = 
	\begin{bmatrix}
		 Z_1^* &0\\
		 G^* &Z_2^*
	\end{bmatrix}
\label{Jform}\mbox{ with }Z_1^* := \begin{bmatrix}
		0 &I\\
		\J{mx}^* &-\alpha_1 I
	\end{bmatrix}
\mbox{ and }
    Z_2^* := \begin{bmatrix}
    	 \J{11}^* &0 &\ldots &0 &0\\
    	 \J{21}^*&0 &\ldots &0 &0\\
    	 \vdots  	& \vdots 	& \vdots 	& \vdots 	& \vdots\\
    	 \J{k1}^* &\J{k2}^* &\ldots &\J{kk}^* &0\\
    	0 &0 &0 &0 &-1
    \end{bmatrix}.
\end{equation}
Therefore the spectrum of $\J{}^*$ is completely determined by $Z_1^*$ and $Z_2^*$. By (\ref{Hlim}) and the spectrum decomposition theorem $\nabla^2E(x^*)=\sum_{i=1}^d\lambda^*_iv_i^*{v_i^*}^\top$, the diagonal block $\J{ii}^*$ of $Z_2^*$ could be expanded as 
\begin{align*}
	\J{ii}^* &= -\zeta\big(\mathcal P_{i-1}^* - v_i^*{v_i^*}^\top \big)\nabla^2E(x^*)+ \zeta({v_i^*}^\top \nabla^2E(x^*)v_i^*I + v_i^*(\nabla^2E(x^*)v_i^*)^\top)\\
&= -\zeta\bigg(  \nabla^2 E(x^*) -  2\sum_{j=1}^i\lambda_j^*v_j^*{v_j^*}^\top -\lambda_i^*I \bigg),
\end{align*} 
which means $J_{ii}^*$ is equipped with the following eigenvalues 
\begin{equation}\label{Jeig}
\left\{ \zeta(\lambda_1^*+\lambda_i^*),\ldots \zeta(\lambda_i^*+\lambda_i^*), \zeta(\lambda_i^* - \lambda_{i+1}^*), \dots , \zeta(\lambda_i^* - \lambda_{d}^*)\right\},\quad 1\leq i\leq k.
\end{equation}
By the assumptions of this theorem, all these eigenvalues are negative such that all eigenvalues of $Z_2^*$ are negative. Then we apply the expansion of $\J{mx}^*$
\begin{equation}\label{Jmx}
	\J{mx}^* = -\alpha_2\bigg( -\sum_{i=1}^k \lambda_i^* v_i^* {v_i^*}^\top + \sum_{j=k+1}^d \lambda_j^* v_j^* {v_j^*}^\top\bigg)
\end{equation}
to partly diagonalize $Z_1^*$ as
\begin{equation}
	\begin{bmatrix}
		V^* &0\\
		0 &V^*
	\end{bmatrix}^\top Z_1^* 
     \begin{bmatrix}
     	V^* &0\\
     	0 &V^*
     \end{bmatrix} = \begin{bmatrix}
      0 &I\\
      D^* &-\alpha_1 I
 \end{bmatrix}=:\tilde{Z}_1^*,
\label{dis1}
\end{equation}
where $V^* = [v_1^*,\ldots, v_d^*]$ is the orthogonal matrix, $D^*:=\mbox{diag}\{d_i^*\}_{i=1}^d$ is the diagonal matrix with the entries  $\{\alpha_2\lambda_1^*,\ldots,\alpha_2\lambda_k^*, -\alpha_2 \lambda_{k+1}^*,\ldots, -\alpha_2\lambda_d^*\}$ and, after a series of permutations, $\tilde{Z}_1^*$ is similar to a block diagonal matrix with $d$ blocks defined by
\begin{equation}
 D_i^* = 
\begin{bmatrix}
	0 &1\\
	d_i^* &-\alpha_1
\end{bmatrix},~~1\leq i\leq d.
\label{dis2}
\end{equation}
The eigenvalues of $D_i^*$ are roots of $x^2 + \alpha_1x -d_i^* = 0$. Since $d_i^*\leq -\alpha_2\mu^*<0$ for $1\leq i\leq d$, we have $\Delta_i = \alpha_1^2 +4d_i^* \leq \alpha_1^2 -4\alpha_2\mu^*\leq 0 $, which implies that the roots of $D_i^*$ are either two same real numbers or conjugate complex numbers. In both cases, the real parts of the roots are negative due to $\alpha_1>0$, which in turn implies that all eigenvalues of $\tilde{Z}_1^*$ have negative real parts. In conclusion, the real parts of all eigenvalues of $\J{}^*$ are negative, which indicates that $(x^*, m^*, v_1^*, v_2^*, ..., v_k^*, l^*)$ is a linearly stable steady state of (\ref{A-HiSD_dyna}) and thud proves (A).

To prove (B) from (A), we suppose that $(x^*, m^*, v_1^*, v_2^*, ..., v_k^*, l^*)$ is a linearly stable steady state of (\ref{A-HiSD_dyna}), which leads to $m^*=l^*=0$ and
\begin{equation}
	(\mathcal P_{i-1}^* - v_i^*{v_i^*}^\top )\nabla^2 E(x^*)v_i^* = 0
	 \label{relation}
\end{equation}
for $1\leq i\leq k$. We intend to apply the mathematical induction to prove
\begin{equation}
	\nabla^2 E (x^*) v_i^*= z_i^*v_i^*,\quad z_i = {v_i^*}^\top \nabla^2 E(x^*)v_i^*\neq 0,\quad {v_i^{*}}^\top v_j^*=0, \quad 1\leq j<i\leq k,
	\label{eigen}
\end{equation}
i.e. $\{(z_i^*,v_i^*)\}_{i=1}^k$ are eigen pairs of $\nabla^2 E(x^*)$.
For (\ref{relation}) with $i=1$,  we have $\nabla^2 E(x^*)v_1^* = z_1^*v_1^*$ with $z_1^* = {v_1^*}^\top \nabla^2 E(x^*)v_1^*$ being an eigenvalue of $\nabla^2 E(x^*)$. Since $\nabla^2 E(x^*)$ is non-degenerate, $z^*_1\not = 0$. (\ref{eigen}) holds true at $i=1$. Assume that (\ref{eigen}) holds for $i=1,2,...,m-1$. Then (\ref{relation}) with $i=m$ yields
\begin{equation}
	\mathcal P_{m-1}^* \nabla^2 E(x^*)v_m^*=\bigg (\nabla^2 E(x^*)  - 2\sum_{j=1}^{m-1}z_j^*v_j^*{v_j^*}^\top\bigg )v_m^* = z_m^*v_m^*,
	\label{i2}
\end{equation} where $z_m^* = {v_m^*}^\top \nabla^2 E(x^*)v_m^*$, which implies $v_m^*$ is the eigenvector of $\mathcal P_{m-1}^* \nabla^2 E(x^*)$. However, since $\{v_j^*\}_{j=1}^{m-1}$ are eigenvectors of $\nabla^2 E (x^*)$, $\nabla^2 E (x^*)$ and $\mathcal P_{m-1}^* \nabla^2 E(x^*)$ share the same eigenvectors such that $v_m^*$ is also an eigenvector of $\nabla^2 E (x^*)$ and there exists an eigenvalue $\mu_m^*$ such that $\nabla^2 E(x^*)v_m^* = \mu_m^*v_m^*$. we multiply ${v_m^*}^\top$ on both sides of this equation and utilize $\|v_m^*\|_2=1$ to obtain ${v_m^*}^\top \nabla^2 E(x^*)v_m^* = \mu_m^*\|v_m^*\|_2^2 = \mu_m^*$, that is, $z_m^*=\mu_m^*$. Combining (\ref{i2}) and $\nabla^2 E(x^*)v_m^* = z_m^*v_m^*$, we have $\sum_{j=1}^{m-1}z_j^*({v_j^*}^\top v_m^* )v_j^* = 0$, and we multiply $(v_s^*)^\top$ for $1\leq s\leq m-1$ on both sides of this equation and apply (\ref{eigen}) to get ${v_s^*}^\top v_m^*=0$ for $ 1\leq s\leq m-1$. Therefore, (\ref{eigen}) holds for $i=m$ and thus for any $1\leq i\leq k$.

As a result of (\ref{eigen}), $\mathcal P_k$ is an orthogonal matrix, which, together with the right-hand side of the dynamics of $m$ in (\ref{A-HiSD_dyna}) with $m^*=0$, indicates $\nabla E (x^*)=0$, i.e. $x^*$ is a critical point of $E$.

It remains to check that $x^*$ is an index-$k$ saddle point of $E$. From $l^*=0$ and $\nabla E (x^*)=0$, $J^*$ could be divided as (\ref{Jform}) and the linear stability condition in (A) implies that all eigenvalues of $Z_1^*$ and $Z_2^*$ have negative real parts. Since (\ref{eigen}) indicates that $\{(z_i^*, v_i^*)\}_{i=1}^k$ are eigen-pairs of $\nabla^2 E(x^*)$, we denote $z_{k+1}^*\leq \ldots \leq z_d^*$ as the remaining eigenvalues of $\nabla^2 E(x^*)$. By the derivation of (\ref{Jmx}), eigenvalues of 
$\J{mx}^*$
are $\{\alpha_2z_1^*,\dots,\alpha_2z_k^*, -\alpha_2z_{k+1}^*,\ldots, -\alpha_2z_d^*\}$. Following the discussions among (\ref{dis1})--(\ref{dis2}), eigenvalues of $Z_1^*$ are roots of $
	x^2 + \alpha_1 x -d_i^*=0$ for $1\leq i\le d$ with $d_i^* =	\alpha_2z_i^*$ for $i\leq k$ and $d_i^*=-\alpha_2z_i^*$ for $k< i\leq d$.
To ensure that all eigenvalues have negative real parts, $d_i^*$ must be negative, which implies $z_i^*<0$ for $i\leq k$ and $z_i^*>0$ for $k<i\leq d$, that is, $x^*$ is an index-$k$ saddle point. 

Finally, it follows from (\ref{Jeig}) that
$\{\zeta(z_i^*+z^*_1),\dots,\zeta(z_i^*+z^*_i), \zeta(z_i^*-z^*_{i+1}),\dots, \zeta(z_i^*$ $-z^*_d)\}$ are eigenvalues of $\J{ii}^*$ for $1\leq i\leq k$. To ensure that all eigenvalues are negative, we have $z^*_1<z^*_2<\dots<z^*_k$ such that $z_i^*$ match $\lambda_i^*$, i.e., $z_i^*=\lambda_i^*$ for $1\leq i\leq k$, which proves (B) and thus completes the proof.
\end{proof}

\section{Local convergence analysis of discrete A-HiSD}
In this section, we present numerical analysis of the discrete Algorithm \ref{hhisd}. Results on some mild conditions indicates that our proposed method has a faster local convergence rate compared with the vanilla HiSD method.
We introduce the following assumptions for numerical analysis, which are standard to analyze the convergence behavior of iterative algorithms \cite{luo2022convergence,nesterov2003introductory}.
\begin{assumption}
	The initial position $x^{(0)}$ locates in a neighborhood of the index-$k$ saddle point $x^*$, i.e., $x^{(0)}\in U(x^*,\delta) =\{x|\|x-x^*\|_2<\delta\}$ for some $\delta>0$ such that
	\begin{enumerate}
		\item[(i)] There exists a constant $M>0$ such that $\|\nabla^2 E(x)-\nabla^2 E(y)\|_2\leq M\|x-y\|_2$ for all $ x,y\in U(x^*,\delta)$;

        \item[(ii)] For any $x\in U(x^*,\delta)$, eigenvalues $\{\lambda_i\}_{i=1}^d$ of $\nabla^2E(x)$ satisfy $\lambda_1\leq\cdots\leq \lambda_k<0< \lambda_{k+1}\leq\cdots\leq\lambda_d$ 
        and there exist positive constants $0<\mu< L$ such that $|\lambda_i|\in[\mu,L]$ for $1\leq i\leq d$.
	
	\end{enumerate}
	\label{original_asm}
\end{assumption}

Assumption \ref{original_asm} leads to an essential corollary on the Lipschitz continuity of the orthogonal projection on eigen subspace, which is stated in Corollary \ref{subspace}. The core of its  proof is the Davis-Kahan theorem \cite{stewart1990matrix}, which captures the sensitivity of eigenvectors under matrix perturbations. Here we adopt an variant of the Davis-Kahan theorem stated as following.
\begin{theorem}\cite[Theorem 2]{10.1093/biomet/asv008}\label{dkthm}
    Let $\Sigma$, $\hat{\Sigma}\in \mathbb{R}^{d\times d}$ be symmetric, with eigenvalues $\lambda_d\geq\ldots\geq\lambda_1$ and $\hat{\lambda}_d\geq\ldots\geq\hat{\lambda}_1$ respectively. Fix $1\leq r \leq s \leq d$ and assume that $\min(\lambda_{r}-\lambda_{r-1}, \lambda_{s+1}-\lambda_{s})>0$, where we define $\lambda_0=-\infty$ and $\lambda_{d+1} = +\infty$.  Let $p=s-r+1$, and let $V=[v_r, v_{r+1},...,v_s]\in\mathbb{R}^{d\times p}$ and $\hat{V}=[\hat{v}_r, \hat{v}_{r+1},...,\hat{v}_s]\in\mathbb{R}^{d\times p}$ have orthonormal columns satisfying $\Sigma v_j=\lambda_j v_j$ and $\hat{\Sigma} \hat{v}_j=\hat{\lambda}_j \hat{v}_j$ for $j=r,r+1,...s$. Then 
    \begin{equation*}
        \|VV^\top - \hat{V}\hat{V}^\top\|_{F} \leq \frac{2\sqrt{2}\min(p^{1/2}\|\hat{\Sigma}-\Sigma\|_2,\|\hat{\Sigma}-\Sigma\|_F)}{\min(\lambda_{r}-\lambda_{r-1}, \lambda_{s+1}-\lambda_{s})},
    \end{equation*}
    where $\|\cdot\|_F$ is the matrix Frobenius norm, i.e. $\|A\|_F = \sqrt{\text{tr}(AA^\top)}$.
\end{theorem}

\begin{corollary}\label{subspace}
    Fix any $x,y\in U(x^*,\delta)$ in Assumption \ref{original_asm}. Denote the orthogonal projections $\mathcal N(x)=\sum_{i=1}^k u_{x,i}u_{x,i}^\top$ and $\mathcal N(y)=\sum_{i=1}^k u_{y,i}u_{y,i}^\top$ where $\{u_{x,i}\}_{i=1}^k$ and $\{u_{y,i}\}_{i=1}^k$ are two sets of orthonormal eigenvectors corresponding to the first k smallest eigenvalues of $\nabla^2 E(x)$ and $\nabla^2 E(y)$, respectively, i.e., $\nabla^2 E(x)u_{x,i}=\lambda_{x,i}u_{x,i}$, $\nabla^2 E(y)u_{y,i}=\lambda_{y,i}u_{y,i}$, $i=1,\ldots,k$, $\{\lambda_{x,i}\}_{i=1}^k$ and $\{\lambda_{y,i}\}_{i=1}^k$ are coresponding eigenvalues. Then
    \[ \|\mathcal N(x) - \mathcal N(y)\|_2 \leq \frac{\sqrt{2k}M}{\mu}\|x-y\|_2. \]
\end{corollary}
\begin{proof}
    We apply Theorem \ref{dkthm} by setting $\Sigma = \nabla^2 E(x)$, $V=[u_{x,1},\ldots, u_{x,k}]$ and $\hat{\Sigma} = \nabla^2 E(y)$, $\hat{V}=[u_{y,1},\ldots, u_{y,k}]$. Then $VV^\top = \mathcal N(x)$ and $\hat{V}\hat{V}^\top = \mathcal N(y)$. By (ii) of Assumption \ref{original_asm}, $\lambda_{x,k}\leq -\mu<\mu\leq\lambda_{x,k+1}$. Hence $\min(\lambda_{x,k+1}-\lambda_{x,k}, \lambda_{x,1}-\lambda_{x,0}) \geq 2\mu$. We have
    \[\|\mathcal N(x) - \mathcal N(y)\|_2 \leq \|\mathcal N(x) - \mathcal N(y)\|_F \leq \frac{2\sqrt{2k}}{2\mu}  \|\nabla^2 E(x) - \nabla^2 E(y)\|_2 \,  \]
    where the first inequality follows because the matrix 2-norm is less than Frobenius norm, and the second from the application of Theorem \ref{dkthm}. We end our proof by (ii) of Assumption \ref{original_asm}.
\end{proof}

We refer the following useful lemma to support our analysis.

\begin{lemma}\cite[Theorem 5]{wang2021modular}
	Let $A: = \begin{bmatrix}
		&(1+\gamma)I-\beta G &-\gamma I\\
		&I &0
	\end{bmatrix}\in \mathbb{R}^{2d\times 2d}$ where $G\in\mathbb{R}^{d\times d}$ is a positive semidefinite matrix. If $\beta\in(0,2/\lambda_{\max}(G))$, then the operator norm of $A^k$ could be bounded as $\|A^k\|_2 \leq D_0(\sqrt{\gamma})^k$ for $1\geq \gamma>\!\max\{(1-\sqrt{\beta\lambda_{\max}(G)})^2, (1-\sqrt{\beta \lambda_{\min}(G)})^2\}$
    where
    \[D_0: = \frac{2(1+\gamma)}{\sqrt{\min\{h(\gamma,\beta\lambda_{\min}(G)),h(\gamma,\beta \lambda_{\max}(G))\}}}\geq 1,\]
   $\lambda_{\max}(G)$ and $\lambda_{\min} (G)$ are the largest and the smallest eigenvalues of G, respectively, and the function $h(\gamma,z)$ is defined as $h(\gamma,z):=-(\gamma-(1-\sqrt{z})^2)(\gamma-(1+\sqrt{z})^2)$.
    \label{mat_vec_bound}
\end{lemma}\par
Based on Lemma \ref{mat_vec_bound}, we derive a generalized result using  the bounds of eigenvalues rather than their exact values.
\begin{corollary}
	Assume that the eigenvalues of $G$ in Lemma \ref{mat_vec_bound} satisfy $0<\mu\leq \lambda_{\min}(G) \leq \lambda_{\max}(G) \leq L$. Then $\|A^k\|_2 \leq C_0(\sqrt{\gamma})^k$ for $1\geq \gamma>\max\{(1-\sqrt{\beta L})^2, (1-\sqrt{\beta \mu})^2\}$ if $\beta\in(0,2/L)$, where $A$ and $h$ are defined in Lemma \ref{mat_vec_bound} and
	\[C_0: = \frac{2(1+\gamma)}{\sqrt{\min\{h(\gamma, \beta\mu), h(\gamma, \beta L)\}}}\geq 1.\]

In particular, choosing 
\begin{equation}\label{para}
 \sqrt{\beta}= \frac{2}{\sqrt{L}+\sqrt{\mu}},~~\sqrt{\gamma}=1-\frac{2}{\sqrt{\kappa}+1}+\varepsilon
\end{equation}
 with $\varepsilon\in(0, \frac{1}{\sqrt{\kappa}+1})$  and $\kappa := L/\mu\geq 1$ leads to
	\begin{equation}\label{c0K}
 C_0\leq\frac{2\sqrt{2}(\sqrt{\kappa}+1)^{1/2}}{\sqrt{3}\varepsilon}=:K.
\end{equation}
	\label{gen_lemma2_2}
\end{corollary}
\begin{proof}
Since $h$ is a concave quadratic function with respect to $z$ and $h(\gamma, z)>0$ if $\gamma$ is chosen under the assumptions and $z\in[\beta\mu,\beta L]$, the minimum of $h(\gamma,\cdot)$ on an interval must occur at end points of this interval. Consequently, $0<\mu\leq \lambda_{\min}(G) \leq \lambda_{\max}(G) \leq L$ leads to
$1\leq D_0 \leq C_0$ where $D_0$ is given in Lemma \ref{mat_vec_bound}.
Furthermore, since $|1-\sqrt{\beta}q|$ is a convex function with respect to $q$, we have
$$\max\{|1-\sqrt{\beta\lambda_{\min}(G)}|, |1-\sqrt{\beta\lambda_{\min}(G)}|\}\leq  \max\{|1-\sqrt{\beta\mu}|, |1-\sqrt{\beta L}|\},$$
 which proves the first statement of this corollary.

We then calculate $h(\gamma,\beta\mu)$ and $h(\gamma,\beta L)$ under the particular choice (\ref{para}) of the parameters, which satisfy the constraints of this corollary by direct calculations. Since 
	\begin{equation*}
		\begin{aligned}
		&
			\gamma - (1-\sqrt{\beta\mu})^2 = (1-\frac{2}{\sqrt{\kappa}+1}+\varepsilon)^2 - (1-\frac{2}{\sqrt{\kappa}+1})^2 = \varepsilon(2-\frac{4}{\sqrt{\kappa}+1}+\varepsilon)\geq \varepsilon^2,\\
    	&\begin{aligned}
    		 (1+\sqrt{\beta\mu})^2 - \gamma = (1+\frac{2}{\sqrt{\kappa}+1})^2-(1-\frac{2}{\sqrt{\kappa}+1}+\varepsilon)^2 &\\
&\hspace{-1in}= (\frac{4}{\sqrt{\kappa}+1}-\varepsilon)(2+\varepsilon)\geq \frac{3(2+\varepsilon)}{\sqrt{\kappa}+1},\end{aligned}
    	\end{aligned}
    \end{equation*}
    where we used $\frac{4}{\sqrt{\kappa}+1}\leq 2$ and $\varepsilon< \frac{1}{\sqrt{\kappa}+1}$,
    we have $h(\gamma,\beta\mu)\geq \varepsilon^2(2+\varepsilon)\frac{3}{\sqrt{\kappa}+1}$. Similarly, 
    \begin{equation*}
    	\begin{aligned}
    		\gamma - (1-\sqrt{\beta L})^2 &= \varepsilon(2-\frac{4}{\sqrt{\kappa}+1}+\varepsilon)\geq \varepsilon^2,\\
    		(1+\sqrt{\beta L})^2 -\gamma &= (2-\varepsilon)(2+\varepsilon+\frac{2\sqrt{\kappa}-2}{\sqrt{\kappa}+1}) \geq (2-\varepsilon)(2+\varepsilon)\geq \frac{3(2+\varepsilon)}{\sqrt{\kappa}+1},
    	\end{aligned}
    \end{equation*}
	which lead to $h(\gamma,\beta L)\geq \varepsilon^2(2+\varepsilon)\frac{3}{\sqrt{\kappa}+1}$. Consequently, $C_0$ could be further bounded as \[C_0 = \frac{2(1+\gamma)}{\sqrt{\min\{h(\gamma,\beta\mu),h(\gamma,\beta L)\}}}\leq\frac{2(1+\gamma)(\sqrt{\kappa}+1)^{1/2}}{\sqrt{3\varepsilon^2(\varepsilon+2)}}\leq \frac{2\sqrt{2}(\sqrt{\kappa}+1)^{1/2}}{\sqrt{3}\varepsilon}, \]
	which completes the proof.
\end{proof} 

We now turn to state our main results. We first reformulate the recurrence relation of the proposed algorithm in Theorem \ref{meta}, based on which we prove the convergence rate of the proposed method in Theorem \ref{main_discrete}. We make the following assumption on Algorithm \ref{hhisd} in the analysis: 

\begin{assumption}\label{eig_assm}
	$\big\{v_i^{(n)}\big\}_{i=1}^k$ are exact eigenvectors corresponding to the first $k$ smallest eigenvalues of $\nabla^2 E(x^{(n)})$ in each iteration in Algorithm \ref{hhisd}, i.e. the error in ``EigenSol'' in Algorithm \ref{hhisd} is neglected.
\end{assumption}

The Assumption \ref{eig_assm} helps to simplify the numerical analysis and highlight the ideas and techniques in the derivation of convergence rates. In practice, approximate eigenvectors computed from the ``EigenSol'' in Algorithm \ref{hhisd} should be considered but leads to more technical calculations in analyzing the convergence rates, see e.g. \cite{luo2022convergence}. This extension will be investigated in the future work.


\begin{theorem}\label{meta}
	The dynamics of $x$ in Algorithm \ref{hhisd} could be written as 
	 \begin{equation}
		\begin{bmatrix}
			 x^{(n+1)}-x^* \\   x^{(n)}-x^*\\
		\end{bmatrix} = 
		\begin{bmatrix}
			(1+\gamma)I - \beta A_* &-\gamma I\\
			 I &0
		\end{bmatrix}
		\begin{bmatrix}
			 x^{(n)}-x^* \\   x^{(n-1)}-x^*\\
		\end{bmatrix} + 
		\begin{bmatrix}
			 p^{(n)} \\
			0
		\end{bmatrix}
		\label{vec_rec}
	\end{equation}
where $	A_* = \mathcal P^*\nabla^2 E(x^*)$, $p_n = \beta(A_* - \tilde{A}_n)(x^{(n)}-x^*)$ and 
$\tilde{A}_n = \mathcal P_n\int_0^1 \nabla^2 E(x^{*} + t(x^{(n)} - x^*))dt$ with $$\mathcal P_{n}:=I-2\sum_{i=1}^kv_i^{(n)}{{}v_i^{(n)}}^\top,~~\mathcal P^*:=I-2\sum_{i=1}^kv_i^{*}{{}v_i^{*}}^\top,$$
where $\{v_i^{*}\}_{i=1}^k$ are eigenvectors corresponding to the first $k$ smallest eigenvalues of $\nabla^2 E(x^{*})$.

Furthermore, if $x^{(n)}\in U(x^*,\delta)$ mentioned in {Assumption \ref{original_asm}} and the Assumption \ref{eig_assm} holds, we have
\begin{equation}
    \|p^{(n)}\|_2\leq\beta C_1\|x^{(n)}-x^*\|^2_2,~~C_1:=\left(\frac{2\sqrt{2k}L}{\mu}+\frac{1}{2}\right
    )M.
	\label{res_estimation}
\end{equation}
\end{theorem}

\begin{proof}
		By the integral residue of the Taylor expansion, we obtain
	\begin{equation*}
		\nabla E(x^{(n)}) - \nabla E(x^{*}) = \left [\int_0^1 \nabla^2 E(x^{*} + t (x^{(n)} - x^*))dt \right ](x^{(n)} - x^*).
	\end{equation*}
As $x^*$ is a critical point, we apply $\nabla E(x^*)=0$ to reformulate the dynamics of $x$ in Algorithm \ref{hhisd} as
	\begin{equation}
		\begin{aligned}
			x^{(n+1)}-x^* & = x^{(n)}-x^*-\beta\mathcal P_n\nabla E(x^{(n)})  + \gamma(x^{(n)} - x^{(n-1)})\\
			& = \left(I - \beta \tilde{A}_n\right)(x^{(n)} - x^{*}) + \gamma(x^{(n)} - x^{(n-1)})\\
			& = \left[(1+\gamma)I - \beta\tilde{A}_n\right](x^{(n)} - x^{*}) - \gamma(x^{(n-1)} - x^{*})
		\end{aligned}
		\label{recur}
	\end{equation}
	where $\tilde{A}_n$ is given in this theorem. We then transform (\ref{recur}) into the matrix form to obtain a single step formula
	\begin{equation}
		\begin{bmatrix}
			 x^{(n+1)}-x^* \\   x^{(n)}-x^*\\
		\end{bmatrix} = 
		\begin{bmatrix}
			(1+\gamma)I - \beta\tilde{A}_n &-\gamma I\\
			 I &0
		\end{bmatrix}
		\begin{bmatrix}
			 x^{(n)}-x^* \\   x^{(n-1)}-x^*\\
		\end{bmatrix}.
		\label{recurr_f}
	\end{equation}
	We finally split $\tilde{A}_n$ as $A_*+(\tilde{A}_n-A_*)$ in this equation to obtain (\ref{vec_rec}).
	
	We turn to estimate $\|p^{(n)}\|_2$ for $x^{(n)}\in U(x^*,\delta)$. As	 $\|p^{(n)}\|_2\leq \beta\|A_*-\tilde{A}_n\|_2\|x^{(n)}-x^*\|_2,
	$
    it suffices to bound $\|A_*-\tilde{A}_n\|_2$, which could be decomposed as 
	\begin{equation*}
		   \|A_*-\tilde{A}_n\|_2  \leq \|A_*-B_n\|_2 + \|B_n -\tilde{A}_n \|_2,~~B_n := \mathcal P_n\nabla^2E(x^*).
	\end{equation*}
	By $\|\mathcal P_n\|_2=1$ and the {\sc Assumption 1}, $\|B_n -\tilde{A}_n \|_2$ could be bounded as follows
	\begin{equation}
		\begin{aligned}
				\|B_n -\tilde{A}_n \|_2&=\|\nabla^2E(x^*) - \int_0^1 \nabla^2 E(x^*+t(x^{(n)}-x^*))dt\|_2\\
				&{\leq}\int_0^1\|\nabla^2E(x^*) - \nabla^2 E(x^*+t(x^{(n)}-x^*))\|_2dt\\
				&{\leq} \int_0^1 t dt M\|x^{(n)}-x^*\|_2=\frac{M}{2}\|x^{(n)}-x^*\|_2,
			\end{aligned}
		\label{b1}
	\end{equation}
where the second inequality is derived from the Lipschitz condition of the Hessian. On the other hand,
$\|A_*-B_n\|_2$ is estimated by using Assumption \ref{eig_assm}
	\begin{equation*}
		\begin{aligned}
				\|A_*-B_n\|_2 & \leq \|\nabla^2 E(x^*)\|_2\|\mathcal P^* -\mathcal P_n \|_2\\
				& = 2\|\nabla^2 E(x^*)\|_2\bigg\|\sum_{i=1}^kv_i^*{v_i^*}^\top -\sum_{i=1}^k\vv{i}\vt{i} \bigg\|_2\leq \frac{2\sqrt{2k}LM}{\mu}\|x^{(n)}-x^*\|_2.
			\end{aligned}
	\end{equation*} 
	The last inequality follows from $\|\nabla^2 E(x^*)\|_2\leq L$ and Corollary \ref{subspace}. Combining the above two equations we get (\ref{res_estimation}) and thus complete the proof.
\end{proof}

Based on the derived recurrence relation, we denote
$$F_{n+1} =
\begin{bmatrix}
	x^{(n+1)}-x^*\\ x^{(n)}-x^*
\end{bmatrix}, \quad
T=\begin{bmatrix}
		(1+\gamma)I - \beta A_* &-\gamma I\\
		 I &0
	\end{bmatrix},\quad
	P_n =  
	\begin{bmatrix}
		p^{(n)}\\ 0
	\end{bmatrix}
	$$
	such that (\ref{vec_rec}) could be simply represented as 
	\begin{equation}
		F_{n+1} = TF_n +P_n.
		\label{short}
	\end{equation}

\begin{theorem}\label{main_discrete}
	Suppose the Assumptions \ref{original_asm} and \ref{eig_assm} hold. Given $\varepsilon\in(0,\frac{1}{\sqrt{\kappa}+1})$, $x^{(-1)}=x^{(0)}$, if  $\|x^{(0)}-x^*\|_2\leq \frac{\hat{r}}{2\sqrt{2} K}$, $\beta $ and $\gamma $ are chosen as (\ref{para}), then $x^{(n)}$ converges to $x^*$ with the estimate on the convergence rate
		\begin{equation}
			\|x^{(n)}-x^*\|_2 \leq 2\sqrt{2}K\|x^{(0)}-x^*\|_2\theta^n,~~n\geq 0,~~\theta:=1-\frac{2}{\sqrt{\kappa}+1}+2\varepsilon,
			\label{main_result}
		\end{equation}
	where $\hat{r}<\min\{\delta, \frac{\sqrt{3}\mu\varepsilon^2(\sqrt{\kappa}+1)^{3/2}}{16\sqrt{2}C_1}\}$, $C_1$ is given in (\ref{res_estimation}) and $K$ is defined in (\ref{c0K}).
\end{theorem}

\begin{proof}
	We derive from (\ref{short}) that 
	\begin{equation}\label{rec}
		F_{n+1} = T^{n+1}F_0 + \sum_{j=0}^n T^{n-j}P_j.
	\end{equation}
   In order to employ Corollary \ref{gen_lemma2_2} to bound $T^m$ for $0\leq m\leq n+1$ in (\ref{rec}), we need to show that $A_*$ is a positive definite matrix. By the spectral decomposition $ \nabla^2E(x^*) = \sum_{i=1}^d \lambda_i^*v_i^*{v_i^*}^\top$ we have 
    \begin{equation}
           	A_* 
           	= \left(I - 2\sum_{i=1}^k v_i^*{v_i^*}^\top\right)\sum_{j=1}^d \lambda_j^*v_j^*{v_j^*}^\top\\
           	=\sum_{i=1}^d h_i^*v_i^*{v_i^*}^\top,
    \end{equation}
    where $h_i^*=-\lambda_i^*$ for $ i=1,...,k$ and $h_j^*=\lambda_j^*$ for $ j=k+1,...,d$. Since $x^*$ is an index-$k$ saddle point, $h_i^*>0$ for $1\leq i\leq d$. By Assumption \ref{original_asm} and Corollary \ref{gen_lemma2_2},
    \begin{equation}
    	\|T^{n+1}F_0\|_2 \leq K(\sqrt{\gamma})^{n+1}\|F_0\|_2,
    	\label{part_1}
    \end{equation}
    and 
    \begin{equation}
    	\Big\|\sum_{j=0}^n T^{n-j}P_j\Big\|_2 \leq \sum_{j=0}^n \|T^{n-j}P_j\|_2\leq \sum_{j=0}^n K (\sqrt{\gamma})^{n-j}\|P_j\|_2=\sum_{j=0}^n K (\sqrt{\gamma})^{n-j}\|p^{(j)}\|_2.
    	\label{power_1}
    \end{equation}
%

To get (\ref{main_result}), it suffices to prove ($\mathbb{A}$)  $\|F_n\|_2\leq 2K\theta^n\|F_0\|_2$ for all $n\geq 0$ with the support of the estimates (\ref{part_1})--(\ref{power_1}). However, in order to bound $p^{(j)}$ in (\ref{power_1}) via (\ref{res_estimation}), which is derived based on the properties in Assumption 1 that are valid only if $\|x^{(j)}-x^*\|_2<\delta$, we also need to show that ($\mathbb{B}$) $\|F_n\|_2\leq \hat{r}(<\delta)$ (such that $\|x^{(n)}-x^*\|_2<\delta$) for all $n\geq 0$. Therefore, we intend to prove the two estimates ($\mathbb{A}$) and ($\mathbb{B}$) simultaneously by induction.

 For $n=0$, both ($\mathbb{A}$) and ($\mathbb{B}$) hold by the assumptions of this theorem and $K\geq C_0\geq 1$ (cf. Corollary \ref{gen_lemma2_2}). Suppose both ($\mathbb{A}$) and ($\mathbb{B}$) hold for $0\leq n\leq m$ for some $m\geq 1$, we obtain from (\ref{power_1}) that
    \begin{equation}
    	\begin{aligned}
    		\Big\|\sum_{j=0}^m T^{m-j}P_j\Big\|_2 
    		&{\leq} \sum_{j=0}^m KC_1\beta(\sqrt{\gamma})^{m-j} \|F_j\|_2^2\\
    		&\hspace{-0.5in}{\leq} 2\beta K^2C_1\hat{r}\sum_{j=0}^m (\sqrt{\gamma})^{m-j}\theta^j \|F_0\|_2{\leq} \frac{2\beta KC_1\hat{r}}{\theta-\sqrt{\gamma}}\cdot K\theta^{m+1}\|F_0\|_2 {\leq} K\theta^{m+1}\|F_0\|_2,
    	\end{aligned}
    \label{part_2}
    \end{equation}
    where in the first inequality we use results in Theorem \ref{meta}, in the second inequality we bound $\|F_j\|_2^2$ by the induction hypotheses ($\mathbb{A}$) and ($\mathbb{B}$), in the third inequality we apply the estimate 
$$\sum_{j=0}^n (\sqrt{\gamma})^{n-j}\theta^j = \theta^n \sum_{j=0}^n \Big(\frac{\sqrt{\gamma}}{\theta}\Big)^{n-j}\leq \frac{\theta^{n+1}}{\theta-\sqrt{\gamma}}$$ 

and in the last inequality we use the assumption $\hat{r}<\min\{\delta, \frac{\sqrt{3}\mu\varepsilon^2(\sqrt{\kappa}+1)^{3/2}}{16\sqrt{2}C_1}\}$ and the definitions of $K$, $\beta$, $\theta$ and $\gamma$ to bound 

$$\frac{2\beta KC_1}{\theta-\sqrt{\gamma}}\hat{r} =\frac{2C_1\hat r}{\varepsilon}\Big(\frac{2}{\sqrt{L}+\sqrt{\mu}}\Big)^2 \frac{2\sqrt{2}(\sqrt{\kappa}+1)^{1/2}}{\sqrt{3}\varepsilon}= \frac{16\sqrt{2}C_1}{\sqrt{3}\mu\varepsilon^2(\sqrt{\kappa}+1)^{3/2}}\hat{r} <1.$$

 Combining (\ref{part_1}), (\ref{part_2}) and $\sqrt{\gamma}<\theta=\sqrt{\gamma}+\varepsilon$, we have 
    $
    	\|F_{m+1}\|_2 \leq 2K\theta^{m+1} \|F_0\|_2
   $ (i.e. the induction hypothesis $(\mathbb{A})$ with $n=m+1$),
  which, together with $\theta<1$ and $\|F_0\|_2=\sqrt{2}\|x^{(0)}-x^*\|_2\leq \frac{\hat r}{2K}$,  leads to 
   $$
   \|F_{m+1}\|_2 \leq 2K\cdot\frac{\hat{r}}{2K}=\hat{r}.
 $$
  That is,  ($\mathbb{A}$) and ($\mathbb{B}$) hold for $n=m+1$ and thus for any $n\geq 0$ by induction, which completes the proof.
\end{proof}

\begin{remark}
If we choose $\varepsilon = \frac{1}{2(\sqrt{\kappa}+1)}$ in Theorem \ref{main_discrete}, the convergence rate $\theta=1-\frac{1}{\sqrt{\kappa}+1}$, which is faster than the $1-\mathcal{O}(\frac{1}{\kappa})$ convergence rate of the standard discrete HiSD proved in \cite{luo2022convergence}, especially for large $\kappa$, i.e. ill-conditioned cases.
\end{remark}

\section{Numerical experiments}
We present a series of numerical examples to demonstrate the efficiency of the A-HiSD algorithm in comparison with the HiSD (i.e. A-HiSD with $\gamma=0$ in Algorithm \ref{hhisd}). All experiments are performed using Python on a laptop with a 2.70-GHz Intel Core i7 processor, 16 GB memory. The same initial point will be used for both methods in each example. 

\subsection{M\"uller-Brown potential}
We start with the well known M\"uller-Brown (MB) potential, a benchmark to test the performance of saddle point searching algorithms. The MB potential is given by \cite{BonKob}
\begin{equation*}
	E_{MB}(x,y) = \sum_{i=1}^4 A_i \exp[a_i(x-\bar{x}_i)^2 + b_i(x-\bar{x}_i)(y-\bar{y}_i) + c_i(y-\bar{y}_i)^2],
\end{equation*}
where $A=[-200, -100, -170, 15]$, $a=[-1, -1, -6.5, 0.7]$, $b=[0, 0, 11, 0.6]$, $c=[-10, -10, -6.5, 0.7]$, $\bar{x}=[1,0,-0.5,-1]$ and $\bar{y}=[0,0.5,1.5,1]$. We show the contour plot of MB potential in Fig.\ref{MB1} with two local minimas marked by red stars and the saddle point connecting them marked by a yellow solid point, as well as and trajectories of A-HiSD with different momentum parameters $\gamma$.  The initial point is $(x^{(0)}, y^{(0)})=(0.15, 1.5)$, the step size $\beta$ is set as $2\times 10^{-4}$ and the SIRQIT is selected as the eigenvector solver.

As shown in Fig.\ref{MB1}, the trajectory of the HiSD method (i.e. the case $\gamma=0$) converges to the saddle point along the descent path. A-HiSD utilizes the momentum term to accelerate the movement along the way such that less configurations appear on the trajectory as we increase $\gamma$. We also show the convergence behavior of four cases in Fig.\ref{MB1}, which indicates that a larger $\gamma$ leads to the faster convergence rate. Since the problem is relatively well-conditioned, we do not need a very large $\gamma$ to accelerate the convergence. This is consistent with our  analysis. 

\begin{figure}[H]
	\centering
	\includegraphics[scale=0.36]{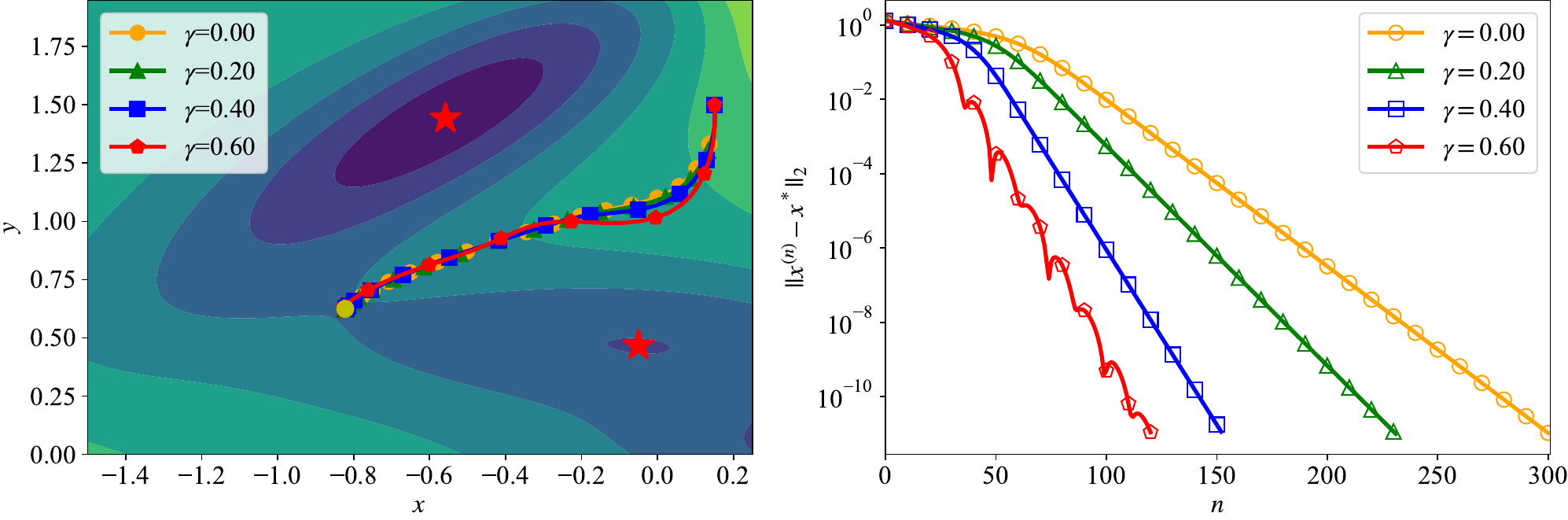}
   \caption{Left: Contour plot of MB potential and trajectories of the A-HiSD with different $\gamma$. For a clear visualization, we plot configurations every 5 steps along the trajectory. Right: Plots of $\|x^{(n)}-x^*\|_2$ with respect to the iteration number $n$. We stop if the error $\|x^{(n)} - x^*\|_2\leq 10^{-11}$.}\label{MB1}
\end{figure}

Next we demonstrate the acceleration of the A-HiSD method by applying the modified MB potential  \cite{BonKob}
\begin{equation*}
	E_{MMB}(x,y) = E_{MB}(x,y) + A_5\sin(xy)\exp[a_5(x-\bar{x}_5)^2 + c_5(y-\bar{y}_5)^2].
\end{equation*}
 The additional parameters are $A_5=500$, $a_5=c_5=-0.1$, $\bar{x}_5=-0.5582$ and $\bar{y}_5=1.4417$. The landscape of modified MB potential is more winding compared with the MB potential. Similarly, we implement the A-HiSD method with different momentum parameters $\gamma$, the initial point $(x^{(0)}, y^{(0)})=(0.053, 2.047)$ and the step size $\beta=10^{-4}$. From Fig.\ref{MB2} we find that the HiSD method (i.e. the case $\gamma=0$) follows the winding ascent path to approach the saddle point, which takes many gradient evaluations along the curved way, while A-HiSD method significantly relax the burden on gradient computations. By utilizing historical information, A-HiSD takes less iterations on the curved way to the saddle point. Furthermore, A-HiSD has the faster convergence rate as shown in Fig. \ref{MB2}.

\begin{figure}[H]
	\centering
	\includegraphics[scale=0.36]{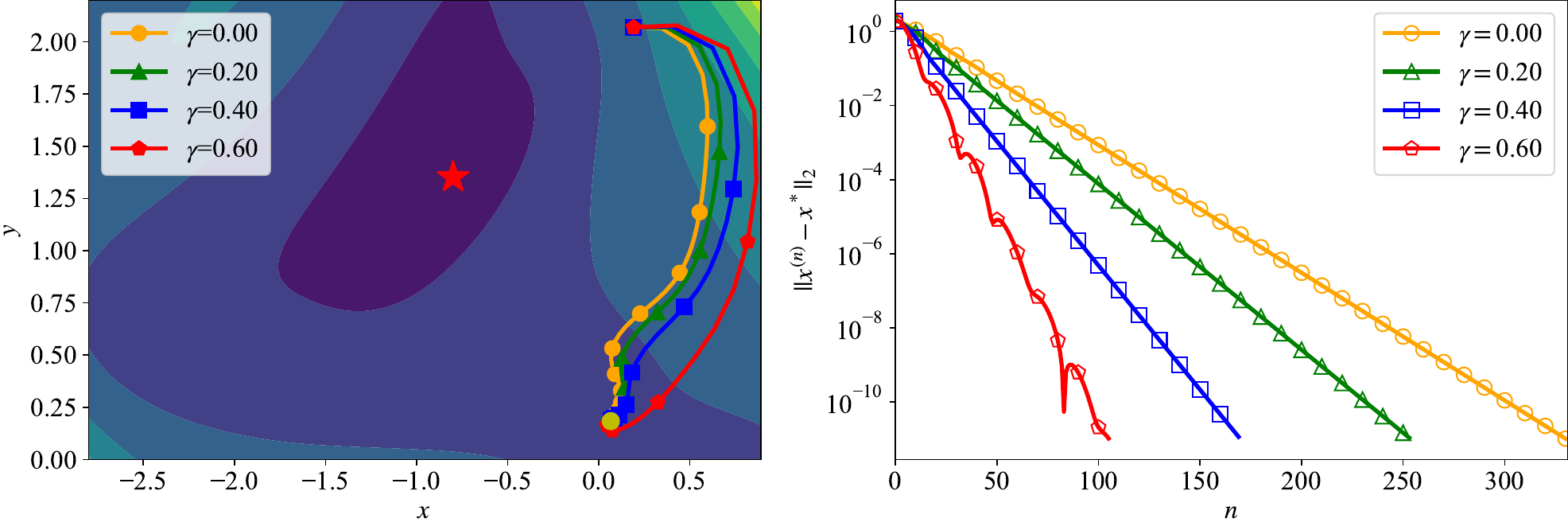}
    \caption{Left: Contour plot of modified MB potential and trajectories of the A-HiSD with different $\gamma$. For a clear visualization, we plot configurations every 5 steps along the trajectory. Right: Plots of $\|x^{(n)}-x^*\|_2$ with respect to the iteration number $n$.  We stop if the error $\|x^{(n)} - x^*\|_2\leq 10^{-11}$.}\label{MB2}
\end{figure}

\subsection{Rosenbrock type function}
The following experiments demonstrate the efficiency of A-HiSD method on high-dimensional ill-conditioned models. We consider the $d$-dimensional Rosenbrock function given by 
\begin{equation*}
	R(x) = \sum_{i=1}^{d-1}[ 100(x_{i+1}-x_i^2)^2 + (1-x_i)^2 ].
\end{equation*}
Here $x_i$ is the $i$th coordinate of the vector $x=[x_1,\cdots, x_d]^\top$. Note that $R(x)$ has a global minimum at $x^* = [1,\cdots,1]^\top$. We then modify $R(x)$ by adding extra quadratic arctangent terms such that $x^*$ becomes a saddle point \cite{luo2022convergence,yin2019high}
\begin{equation}
	R_m(x) = R(x) + \sum_{i=1}^d s_i\arctan^2(x_i - x^*_i).
\end{equation}
We set $d=1000$, $s_j = 1$ for $j>5$ and (i) $s_i=-500$ or (ii) $s_i=-50000$ for $1\leq i\leq 5$. We set $x^{(0)} = x^* + \rho \frac{n}{\|n\|_2}$ as the initial point where $n\sim \mathcal{N}(0, I_d)$ is a normal distribution vector and $\rho$ is set as $1.0$ for the case (i) and $0.1$ for the case (ii). 
 
We implement the A-HiSD under different $\gamma$ to compare the numerical performance. Convergence behaviors of the case (i) are shown in Fig.\ref{rosen1}. In this case, $x^*$ is an index-3 saddle point and the condition number $\kappa^*$ of $\nabla^2 R_m(x^*)$ is 721.93, which can be regraded as the approximation of $\kappa$ mentioned in Theorem \ref{main_discrete}. The step size $\beta$ is set as $2\times 10^{-4}$ and the eigenvector computation solver is set as the SIRQIT. Fig.\ref{rosen1} indicates that the HiSD (i.e. $\gamma=0$) has a slow converge rate, while increasing $\gamma$ leads to a faster convergence rate. The left part of Fig. \ref{rosen1} shows that $\gamma=0.6$ leads to the fastest convergence rate when $x^{(n)}$ is relatively far from $x^*$. Over increasing $\gamma$ may not be beneficial for the acceleration in this situation. As $x^{(n)}$ approaches $x^*$, the A-HiSD with $\gamma=0.95$ outperforms other cases, as shown in the right part of Fig.\ref{rosen1}. It achieves the stop criterion $\|x^{(n)} - x^*\|_2\leq 10^{-10}$  within 2000 iterations. In comparison, the case $\gamma=0$ requires more than 30000 iterations.
	
\begin{figure}[htbp]
	\centering
        \includegraphics[scale=0.36]{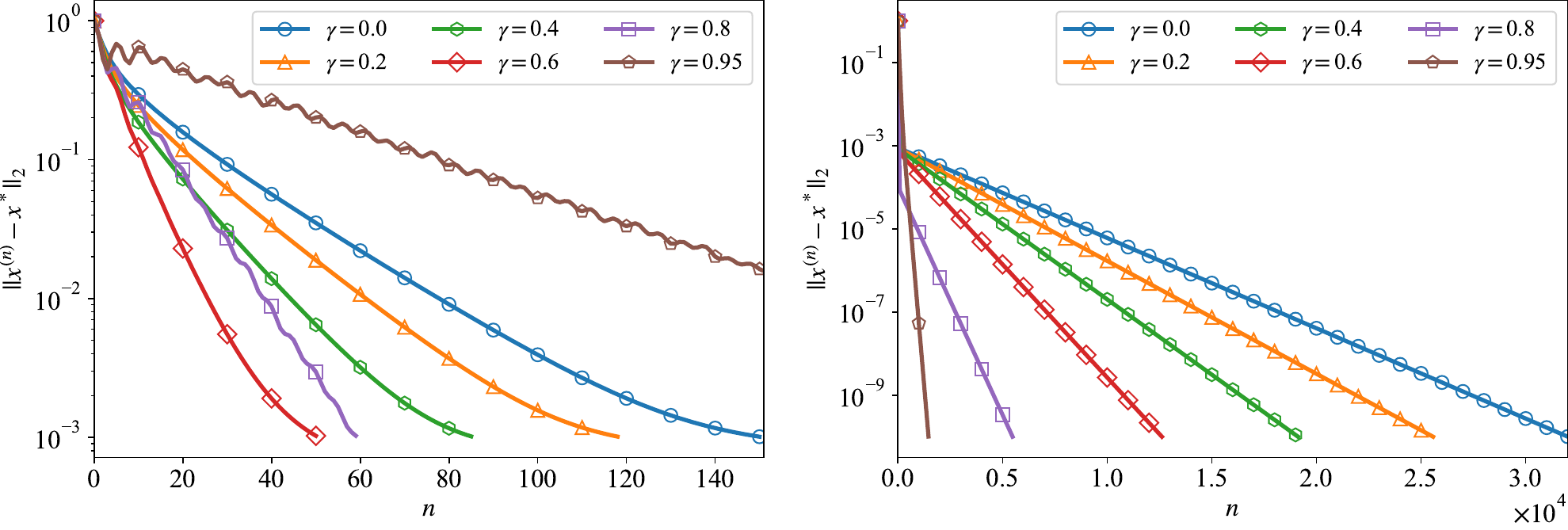}
	\caption{Plots of $||x^{(n)}-x^*||_2$ with respect to the iteration number $n$ for the modified Rosenbrock function with the condition (i). Condition number $\kappa^*$ at the index-3 saddle point $x^*$ is 721.93. Left: Stop if the error $\|x^{(n)} - x^*\|_2\leq 10^{-3}$. Right: Stop if the error $\|x^{(n)} - x^*\|_2\leq 10^{-10}$.}\label{rosen1}
\end{figure}

\begin{figure}[htbp]
	\centering
        \includegraphics[scale=0.36]{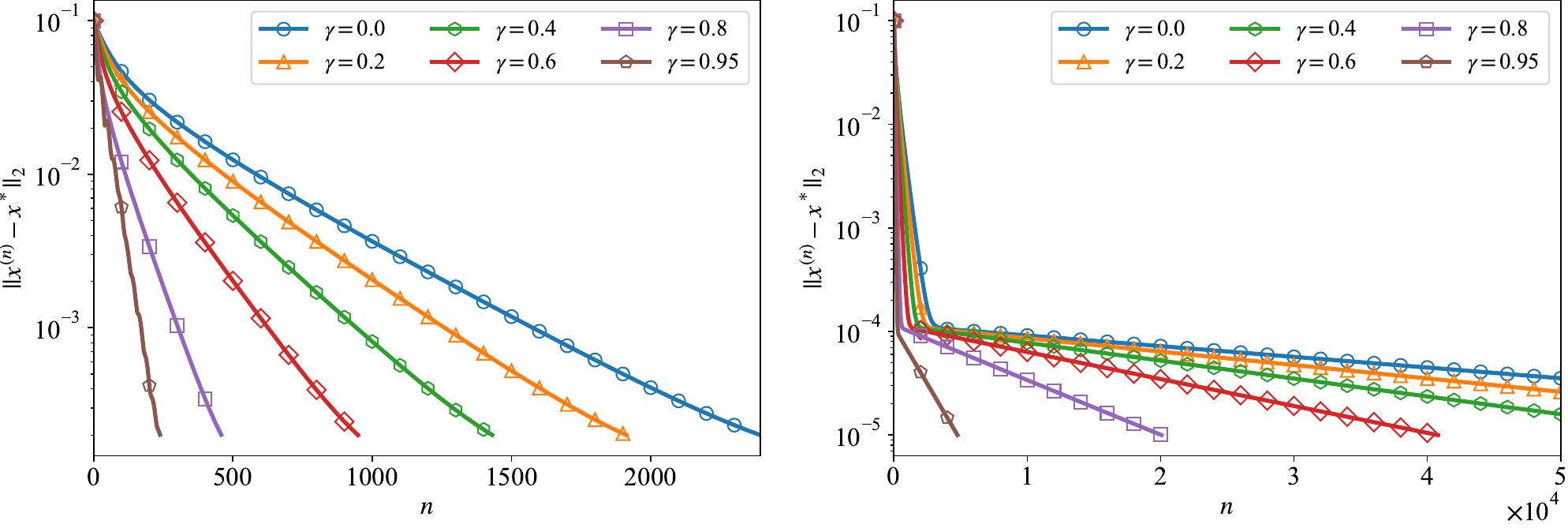}
	\caption{Plots of $||x^{(n)}-x^*||_2$ with respect to the iteration number $n$ for the modified Rosenbrock function with the condition (i). Condition number $\kappa^*$ at the index-5 saddle point $x^*$ is 39904.30. Left: Stop if the error $\|x^{(n)} - x^*\|_2\leq 2\times10^{-4}$. Right: Stop if the error $\|x^{(n)} - x^*\|_2\leq 10^{-5}$ or terminate when $n=50000$.}\label{rosen2}
\end{figure}

Results of the case (ii) are shown in Fig.\ref{rosen2}. In this case, $\nabla^2 R_m(x^*)$ is much more ill-conditioned compared with the case (i) since the condition number $\kappa^*$ is 39904.29. Due to the ill-condition, we have to choose a smaller step size $\beta=10^{-5}$ and select LOBPCG as the eigenvector solver.
 Meanwhile, $x^*$ becomes an index-5 saddle point. All these settings make the case (ii) more challenging. We observe from Fig.\ref{rosen2} that the convergence rate of  the HiSD method is very slow, while the momentum term helps to release the problem by introducing the historical gradient information. For instance, A-HiSD with $\gamma=0.95$ attains the stop criterion $\|x^{(n)}-x^*\|_2\leq 10^{-5}$ within only 6000 iterations. Furthermore, as discussed in our numerical analysis, larger $\gamma$ should be applied in ill-conditioned cases, which is consistent with the numerical observations.

\subsection{Modified strictly convex 2 function}
We compare the performance of A-HiSD method with another commonly-used acceleration strategy, i.e. the HiSD with the Barzilai-Borwein (BB) step size \cite{barzilai1988two}, which could accelerate the HiSD according to experiment results in \cite{zhang2016optimization,yin2019high}. 
We apply the modified strictly convex 2 function, a typical optimization test example proposed in \cite{raydan1997barzilai}
\begin{equation}
	S(x) = \sum_{i=1}^d s_ia_i(e^{x_i}-x_i)/10,
\end{equation}
where $a_i = 1+5(i-1)$ for $1\leq i\leq d$, $s_i=-1$ for $1\leq i\leq k$ and $s_j = 1$ for $k+1\leq j\leq d$. $x^* = [0,...,0]$ is an index-$k$ saddle point of $S(x)$. In our experiment, we set $d=100$ and $k=5$.

 We implement A-HiSD method with different momentum parameters $\gamma$ and compare with the HiSD with the BB step size $\beta_n$ in each iteration
\begin{equation}
	\beta_n = \min\left\{\frac{\tau}{\|\nabla S(x^{(n)})\|_2}, \frac{{ \Delta x^{(n)}}^\top \Delta g^{(n)}}{\|\Delta g^{(n)}\|_2^2}\right\},
	\label{bb_step_size}
\end{equation}
where $\Delta x^{(n)} = x^{(n)} - x^{(n-1)}$. $\Delta g^{(n)} = \nabla S(x^{(n)}) - \nabla S(x^{(n-1)})$ and $\tau=0.5$ to avoid the too large step size. SIRQIT is selected as the eigenvector solver. The initial point $x^{(0)}$ is set as $[-6,...,-6]$. 
\begin{figure}[h]
	\centering 
	\includegraphics[scale=0.45]{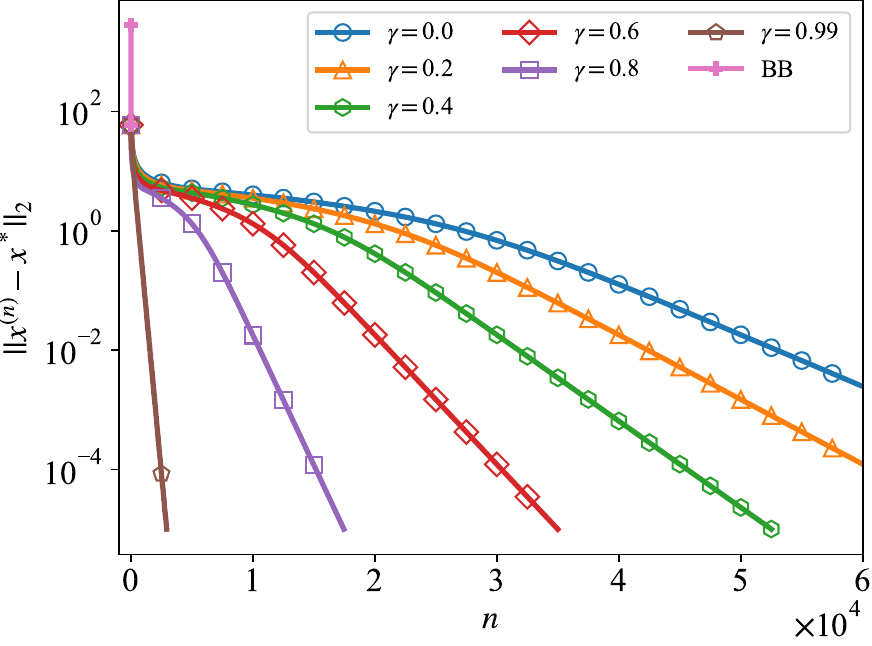}\vspace{-0.1in}
	\caption{Plots of $||x^{(n)}-x^*||_2$ with respect to the iteration number $n$ for the modified strictly convex 2 function. BB represents the HiSD method with the BB step size. We stop if the error $\|x^{(n)} - x^*\|_2\leq10^{-5}$ or terminate if $n=60000$.}
	\label{sc2}
\end{figure}

As shown in Fig.\ref{sc2}, the HiSD with the BB step size diverges within a few iterations, while the A-HiSD method still converges. 
Meanwhile, as we increase $\gamma$, the empirical convergence rate of the A-HiSD increases, which substantiates the effectiveness and the stability of the A-HiSD method in comparison with the HiSD with the BB step size.

\subsection{Loss function of neural network}
We implement an interesting, high-dimensional example of searching saddle points on the neural network loss landscape. {Due to the non-convex nature of loss function, multiple local minima and saddle points are main concerns of the neural network optimization \cite{dauphin2014identifying}. Recent research indicates that gradient-based optimization methods with small initializations can induce a phenomenon known as 'saddle-to-saddle' training process in various neural network architectures, including linear networks and diagonal linear networks \cite{NEURIPS2022_7eeb9af3,jacot2021saddle,pesme2023saddle}. In this process, network parameters may become temporarily stuck before rapidly transitioning to acquire new features. Additionally, empirical evidence suggests that vanilla stochastic gradient methods can converge to saddle points when dealing with class-imbalanced datasets—common occurrences in real-world datasets \cite{NEURIPS2022_8f4d70db}. Consequently, it is urgent to figure out the impact of saddle points in deep learning. Analyzing saddle point via computations can provide valuable insights into specific scenarios. } 

Because of the overparameterization of neural networks, most critical points of the loss function are highly degenerate, which contains many zero eigenvalues. This challenges saddle searching algorithms since it usually leads to the slow convergence rate in practice. Let $\{(x_i, y_i)\}_{i=1}^m$ with $x_{i}\in\mathbb{R}^{d_x}$ and $y_i\in\mathbb{R}^{d_y}$ be the training data. We consider a fully-connected linear neural network of depth $H$
\begin{equation}
	f_{linear}(\textbf{W}; x) = W_HW_{H-1}\cdots W_{2}W_{1}x,
\end{equation}
where $\textbf{W} = [W_1,W_2,...,W_H]$ with weight parameters $W_{i+1}\in\mathbb{R}^{d_{i+1}\times d_i}$ for $0\leq i \leq H-1$ with $d_0=d_x$ and $d_{H}=d_y$. The corresponding empirical loss $L$ is defined by \cite{Ach}
\begin{equation}
	L(\textbf{W}) = \sum_{i=1}^m\| f_{linear}(\textbf{W};x_i) - y_i\|_2^2 = \|W_HW_{H-1}\cdots W_{2}W_{1}X-Y\|_F^2. 
\end{equation} 
where $X=[x_1,...,x_m]$ and $Y=[y_1,...,y_m]$. In practice we could vectorize $W_1,\cdots,W_H$ row by row such that $L$ can be regarded as a function defined in $\mathbb{R}^{D}$ with $D = \sum_{i=0}^{H-1} d_id_{i+1}$. According to \cite{Ach}, $L$ contains a large amount of saddle points, which can be parameterzied by the data set $\{(x_i, y_i)\}_{i=1}^m$. For instance, if $d_i=d_0$ for $1\leq i \leq H-1$, then $\textbf{W}^*:=[W_1^*,\ldots,W_H^*]$ is a saddle point where
\begin{equation}
	W_1^*=\begin{bmatrix}
		 U_{\mathcal{S}}^\top\Sigma_{YX}\Sigma_{XX}^{-1}\\
		 0
	\end{bmatrix},\quad 
   W_h^* = I\text{ for }2\leq h\leq H-1,\text{ and }
    W_H^* = \begin{bmatrix}
    	U_{\mathcal{S}} &0
    \end{bmatrix}.
\label{cp}
\end{equation}
 Here $I$ is the identity matrix, $\mathcal{S}$ is an index subset of $\{1,2,...,r_{\max}\}$ for $r_{\max}=\min\{d_0,...,d_H\}$, 
 $\Sigma_{XX}=XX^\top$, $\Sigma_{YX}=YX^\top$, $\Sigma=\Sigma_{YX}\Sigma_{XX}^{-1}\Sigma_{YX}^\top$ and $U$ satisfies $\Sigma = U\Lambda U^\top$ with $\Lambda=\text{diag}(\lambda_1,...,\lambda_{d_y})$. It is assumed in \cite{Ach} that $\lambda_1>...>\lambda_{d_y}> 0$, which holds true under the data of this experiment. $U_\text{S}$ is then obtained by concatenating column vectors of $U$ according to the index set $\mathcal{S}$. 

\begin{figure}[htbp]
	\centering
	\includegraphics[width=6cm]{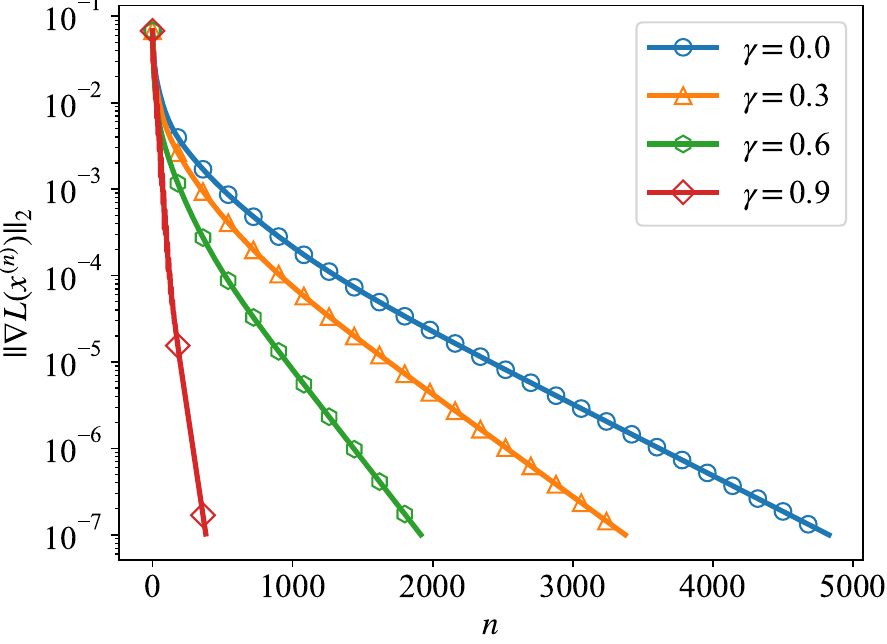}
	\caption{Plots of the gradient norm $||\nabla L(x^{(n)})||_2$ with respect to the iteration number $n$ for the loss function of the neural network. We stop if the error $\|L(x^{(n)})\|_2\leq10^{-7}$.}
	\label{linear_nn} 
\end{figure}

In this experiments, we set the depth $H=5$, the input dimension $d_x=10$, the output dimension $d_y=4$, $d_i=10$ for $1\leq i\leq 4$, and the number of data points $m=100$. Data points $(x_i,y_i)$ are drawn from the normal distributions $\mathcal{N}(0,I_{d_x})$ and $\mathcal{N}(0,I_{d_y})$, respectively. The initial point is $(W_1^{(0)},\cdots, W_H^{(0)}) = \textbf{W}^* + (V_1,\cdots, V_H)$ where $(V_1,\cdots, V_H)$ is a random perturbation whose elements $(V_h)_{i,j}$ are drawn from $\mathcal{N}(0,\sigma^2_h)$ independently, with $\sigma_h = 0.5\frac{\|W_h^*\|_F}{\sqrt{d_{h-1}d_h}}$. Under the current setting,  $\textbf{W}^*$ is a degenerate saddle point with 16 negative eigenvalues and several zero eigenvalues. Although A-HiSD is developed under the framework of non-degenerate saddle points, we show that this algorithm works well for the degenerate case. We set the step size $\beta=0.1$ and the LOBPCG is selected as the eigenvector solver. Due to the highly degenerate property of the loss landscape, $\textbf{W}^*$ is not an isolated saddle point but lies on a low-dimensional manifold. Hence we compute the gradient norm instead of the Euclidean distance as the accuracy measure.

From Fig.\ref{linear_nn} we find that the A-HiSD with $\gamma=0.9$ attains the tolerance $10^{-7}$ within 500 iterations, while that the HiSD (i.e. the case $\gamma=0$) takes about 5000 iterations.  In Table \ref{t_linear}, we report the number of iterations (denoted as ITER) and the computing time in seconds (denoted as CPU) of the the problem in details for the comparison. A faster convergence rate is achieved as we gradually increase $\gamma$, which indicates the great potential of A-HiSD method on the acceleration of convergence in highly degenerate problems.

\begin{table}[htbp]
\caption{The ITER and CPU of HiSD ($\gamma=0$) and A-HiSD (with $\gamma=0.3, 0.6, 0.9$) for searching degenerate index-16 saddle points of the loss function of the linear neural networks. The speedup equals to (CPU of HiSD)/(CPU of A-HiSD).}
\resizebox{\textwidth}{!}{
\begin{tabular}{cc|ccc|ccc|ccc}
\hline
\multicolumn{2}{c|}{$\gamma=0$}   & \multicolumn{3}{c|}{$\gamma=0.3$}                                & \multicolumn{3}{c|}{$\gamma=0.6$}                                & \multicolumn{3}{c}{$\gamma=0.9$}                                \\ \hline
\multicolumn{1}{c|}{CPU}   & ITER & \multicolumn{1}{c|}{CPU}   & \multicolumn{1}{c|}{ITER} & speedup & \multicolumn{1}{c|}{CPU}   & \multicolumn{1}{c|}{ITER} & speedup & \multicolumn{1}{c|}{CPU}  & \multicolumn{1}{c|}{ITER} & speedup \\ \hline
\multicolumn{1}{c|}{56.22} & 4830 & \multicolumn{1}{c|}{41.80} & \multicolumn{1}{c|}{3376} & 1.33    & \multicolumn{1}{c|}{25.77} & \multicolumn{1}{c|}{1917} & 2.15    & \multicolumn{1}{c|}{6.62} & \multicolumn{1}{c|}{382}  & 8.35    \\ \hline
\end{tabular}\label{t_linear}
}
\end{table}

\section{Concluding remarks}
We present the A-HiSD method, which integrates the heavy ball method with the HiSD method to accelerate the computation of saddle points. By employing the straightforward update formulation from the heavy ball method, the A-HiSD method achieves significant accelerations on various ill-conditioned problems without much extra computational cost. We establish the theoretical basis for A-HiSD at both continuous and discrete levels. Specifically, we prove the linear stability theory for continuous A-HiSD and rigorously prove the faster local linear convergence rate of the discrete A-HiSD in comparison with the HiSD method. These theoretical findings provides strong supports for the convergence and acceleration capabilities of the proposed method.

While we consider the A-HiSD method for the finite-dimensional gradient system in this paper,  it has the potential to be extended to investigate the non-gradient systems, which frequently appear in chemical and biological systems such as gene regulatory networks \cite{QIAO2019271}. Furthermore, this method can be adopted to enhance the efficiency of saddle point search for infinite-dimensional systems, such as the Landau-de Gennes free energy in liquid crystals \cite{han2021solution,Shi_2023}.
Lastly, the effective combination of HiSD and the heavy ball method inspires the integration of other acceleration strategies with HiSD, such as the Nesterov accelerated gradient method \cite{nesterov1983method} and the Anderson mixing \cite{anderson1965iterative}. We will investigate this interesting topic in the near future.
\\
\\

{\small 
	\noindent\textbf{Funding}  This work was supported by National Natural Science Foundation of China (No.12225102, T2321001, 12050002, 12288101 and 12301555), and the Taishan Scholars Program of Shandong Province.\\
	
	\noindent\textbf{Data Availability} The datasets generated and/or analysed during the current study are available from the corresponding author on reasonable request\\
}
\section*{Declarations}
	\noindent\textbf{Conflict of interest} The authors have no relevant financial interest to disclose.

\bibliographystyle{spmpsci}
\bibliography{reference}
\end{document}